%% file: main.tex
\documentclass[12pt]{amsart}

\usepackage[margin = 1in]{geometry}
\usepackage{amsfonts, amsmath,amscd, amssymb,
latexsym, mathrsfs, mathtools, 
slashed, stmaryrd, verbatim, wasysym, bbm, url}
\usepackage[usenames]{xcolor}
\usepackage{enumerate}
\usepackage[shortlabels]{enumitem}
\usepackage[utf8]{inputenc}
\usepackage[colorlinks,citecolor=blue,urlcolor=red,bookmarks=false,hypertexnames=true]{hyperref} 
\hypersetup{
  colorlinks,
  citecolor=blue,
  linkcolor=blue,
  urlcolor=blue}

\usepackage[backend=bibtex,bibstyle=authoryear,citestyle=alphabetic,firstinits=true,doi=false,isbn=false,url=false,backref]{biblatex}
\DeclareFieldFormat{labelalphawidth}{\mkbibbrackets{#1}}

\defbibenvironment{bibliography}
  {\list
     {\printtext[labelalphawidth]{%
        \printfield{prefixnumber}%
    \printfield{labelalpha}%
        \printfield{extraalpha}}}
     {\setlength{\labelwidth}{\labelalphawidth}%
      \setlength{\leftmargin}{\labelwidth}%
      \setlength{\labelsep}{\biblabelsep}%
      \addtolength{\leftmargin}{\labelsep}%
      \setlength{\itemsep}{\bibitemsep}%
      \setlength{\parsep}{\bibparsep}}%
      }
  {\endlist}
  {\item}

\newtheorem*{acknowledgements}{Acknowledgements}
\newtheorem*{theorem*}{Theorem}
\newtheorem{theorem}{Theorem}[section]
\newtheorem{corollary}{Corollary}[section]

\newtheorem{lemma}{Lemma}[section]

\theoremstyle{definition}
\newtheorem{example}{Example}
\newtheorem{remark}[example]{Remark}

\numberwithin{equation}{section}

\let\oldsqrt\sqrt
\def\sqrt{\mathpalette\DHLhksqrt}
\def\DHLhksqrt#1#2{%
\setbox0=\hbox{$#1\oldsqrt{#2\,}$}\dimen0=\ht0
\advance\dimen0-0.2\ht0
\setbox2=\hbox{\vrule height\ht0 depth -\dimen0}%
{\box0\lower0.4pt\box2}}

\usepackage{verbatim} 
\DeclareFontFamily{U}{mathx}{\hyphenchar\font45}
\DeclareFontShape{U}{mathx}{m}{n}{
      <5> <6> <7> <8> <9> <10>
      <10.95> <12> <14.4> <17.28> <20.74> <24.88>
      mathx10
      }{}
\DeclareSymbolFont{mathx}{U}{mathx}{m}{n}
\DeclareFontSubstitution{U}{mathx}{m}{n}
\DeclareMathAccent{\widecheck}{0}{mathx}{"71}

\newcommand\eps\varepsilon
\renewcommand\epsilon\varepsilon

\newcommand{\abs}[1]{\left\lvert #1 \right\rvert}

\newcommand{\norm}[1]{\lVert #1 \rVert}

\newcommand\dist{\operatorname{dist}}

\newcommand\Mand{\text{ and }}

\newcommand\Mfor{\text{ for }}

\newcommand\paperintro%
        {%
         }
\newcommand\paperbody%
        {%
         }

\newcommand\bbG{\mathbb{G}}
\newcommand\bbH{\mathbb{H}}

\newcommand\bbR{\mathbb{R}}

\newcommand\cA{\mathcal{A}}

\newcommand\cD{\mathcal{D}}

\newcommand\cF{\mathcal{F}}

\newcommand\cH{\mathcal{H}}

\newcommand\mf[1]{\mathfrak{ #1}}

\DeclareMathAlphabet{\mathpzc}{OT1}{pzc}{m}{it}

\newcommand{\sbs}{\subset}

\usepackage[refpage]{nomencl}

\makenomenclature

\usepackage{tikz, tikz-cd, tkz-euclide}
\usetikzlibrary{calc,positioning,intersections,quotes,decorations.markings}
\makeatletter
\def\@tocline#1#2#3#4#5#6#7{\relax
  \ifnum #1>\c@tocdepth 
  \else
    \par \addpenalty\@secpenalty\addvspace{#2}%
    \begingroup \hyphenpenalty\@M
    \@ifempty{#4}{%
      \@tempdima\csname r@tocindent\number#1\endcsname\relax
    }{%
      \@tempdima#4\relax
    }%
    \parindent\z@ \leftskip#3\relax \advance\leftskip\@tempdima\relax
    \rightskip\@pnumwidth plus4em \parfillskip-\@pnumwidth
    #5\leavevmode\hskip-\@tempdima
      \ifcase #1
       \or\or \hskip 1em \or \hskip 2em \else \hskip 3em \fi%
      #6\nobreak\relax
    \hfill\hbox to\@pnumwidth{\@tocpagenum{#7}}\par
    \nobreak
    \endgroup
  \fi}
\makeatother

\def\annu#1{_{%
  \vbox{\hrule height .2pt 
    \kern 1pt 
    \hbox{$\scriptstyle {#1}\kern 1pt$}%
  }\kern-.05pt 
  \vrule width .2pt 
}}

\allowdisplaybreaks
\title[Bakry-\'Emery for weighted contractivity and $L^2$-Hardy inequalities]{A Bakry-\'Emery criterion for weighted contractivity and $L^2$-Hardy inequalities}
\author{Yaozhong Qiu}
\address{Department of Mathematics, Imperial College London, 180 Queen’s
Gate, London, SW7 2AZ, United Kingdom}
\email{y.qiu20@imperial.ac.uk}

\begin{document}

\begin{abstract}
We show a symmetric Markov diffusion semigroup satisfies a weighted contractivity condition if and only if a $L^2$-Hardy inequality holds, and we give a Bakry-\'Emery type criterion for the former. We then give some applications. 
\end{abstract}

\maketitle

\input{chapters/1-introduction}
\input{chapters/2-theorems}
\input{chapters/3-applications}

\begin{acknowledgements}
\textup{The author would like to thank Boguslaw Zegarli\'nski for introducing the author to the problem, his warm encouragment, and for helpful discussions, and Pierre Monmarch\'e and Cyril Roberto for valuable discussions and suggestions. We are also very grateful to and would like to deeply thank the anonymous referee for their thorough reading and comments which helped improve the presentation of the paper. The author is supported by the President's Ph.D. Scholarship of Imperial College London.}
\end{acknowledgements}

\printbibliography

\end{document}

%% file: chapters/1-introduction.tex
\section{Introduction}

In this paper, we prove $L^2$-Hardy inequalities from the viewpoint of symmetric Markov semigroups $(P_t)_{t \geq 0}$ generated by a diffusion $L$ with carr\'e du champ $\Gamma$ and reversible measure $\mu$. Our starting point is the general theory of Markov semigroups which associates a property (or properties) of the semigroup with a functional inequality, for instance the hypercontractivity and exponentially fast entropic convergence to equilibrium of the semigroup with the logarithmic Sobolev inequality. Here we shall associate the titular weighted contractivity of the semigroup with the Hardy inequality, and give a sufficient condition for the former following the ideas of \cite{RZ21} where the authors studied the generalised Bakry-\'Emery calculus for a carr\'e du champ with an additional multiplicative term. 

The classical Bakry-\'Emery calculus is typically associated with the study of semigroups generated by elliptic diffusions, for instance when $L = \Delta - \nabla U \cdot \nabla$ is a (euclidean) Kolmogorov type operator on $\bbR^m$ where $\nabla$ and $\Delta$ are respectively the euclidean gradient and laplacian and $U: \bbR^m \rightarrow \bbR$ is a suitable real-valued function. When $L$ is nonelliptic however, for instance when $\nabla$ and $\Delta$ are respectively replaced with the subgradient $\nabla_{\bbH^m}$ and sublaplacian $\Delta_{\bbH^m}$ on the Heisenberg group $\bbH^m$, the classical curvature(-dimension) condition is violated (see, for instance, \cite[pp.~3]{Bak08}). Actually, the classical Bakry-\'Emery calculus is already strong enough to recover the Hardy inequality, at least in the euclidean setting. Indeed, it was proved in \cite{DGZ21} that the Hardy inequality, which can be regarded as a special case of the Caffarelli-Kohn-Nirenberg inequality 
\[ \left(\int_{\bbR^m} \frac{\abs{f}^p}{\abs{x}^{bp}}dx\right)^{2/p} \leq C_{a, b}\int_{\bbR^m} \frac{\abs{\nabla f}^2}{\abs{x}^{2a}}dx, \] 
can be realised as a special case of the Sobolev inequality 
\[ \norm{f}_p^2 \lesssim \norm{\nabla f}^2_2 \]
on a weighted manifold, and therefore can be recovered by the classical Bakry-\'Emery calculus (see, for instance, \cite[\S2.3]{Bak04} or \cite[\S6]{BGL14}). 

In the nonelliptic setting, the theory of generalised Bakry-\'Emery calculi has bore fruit. Roughly speaking, the classical Bakry-\'Emery calculus studies the ``gamma calculus" or ``$\Gamma_2$-calculus" of the carr\'e du champ operator $\Gamma$, whereas generalised Bakry-\'Emery calculi study the gamma calculi of other (perhaps similar) operators. For instance, \cite{BB12, BG16} developed a generalised curvature-dimension condition adapted to the subelliptic (or subriemannian) setting and proved analogues of the Poincar\'e and logarithmic Sobolev inequalities by studying a carr\'e du champ with an additional gradient term containing derivatives in the so-called ``vertical" direction. \cite{Bau17, baudoin2021gamma} used a similar idea to study the kinetic Fokker-Planck equation and recovered hypocoercive type estimates for convergence to equilibrium which appeared in \cite{Vil06}. We refer the reader to \cite{Mon19} for more details on generalised Bakry-\'Emery calculi and their applications. Our analysis falls under the scope of the framework of \cite{Mon19} and, in the language of that paper, corresponds to the study of $\Phi^W(f) = W^2f^2$. Note $\Phi^W$ contains none of the traditional gradient fields, but it can still be regarded as a generalised carr\'e du champ associated to a differential operator of order zero.

The contents of this paper are organised as follows. In \S3, we study the generalised Bakry-\'Emery calculus, denoted by $\Gamma^W$, for $\Phi^W$, and prove a correspondence between the curvature condition $\Gamma^W(f) \geq \gamma \Phi^W(f)$ and the pointwise subcommutation $\Phi^W(P_tf) \leq e^{-2\gamma t} P_t \Phi^W(f)$. We then prove a correspondence between the integrated subcommutation 
\[ \int \Phi^W(P_tf)d\mu \leq e^{-2\gamma t} \int \Phi^W(f)d\mu \]
and the functional inequality 
\[ \int L(W^2)f^2d\mu \leq 2\int W^2\Gamma(f)d\mu - 2\gamma \int W^2f^2d\mu, \]
that is, for $\gamma = 0$, a correspondence between the weighted contractivity of the semigroup and the Hardy inequality. Following this, we restate a sufficient condition first appearing in \cite{RZ21} for the curvature condition, which sequentially implies the pointwise subcommutation, the integrated subcommutation, and finally the Hardy inequality. In \S4, we present some applications of our results in various settings. We refer the reader to \cite{Dav99, BEL15} for a general review of Hardy inequalities and their applications. 

\section{Preliminaries}
An abstract definition of a diffusion can be given in the absence of a differentiable structure together with an abstract notion of regularity, but we shall restrict our attention momentarily here to $\bbR^m$ where a diffusion $L$ is defined as a second order differential operator of the form
\begin{equation}\label{intro:L}
    L = \sum_{i,j=1}^m a_{ij}\partial_{ij} + \sum_{i=1}^m b_i\partial_i
\end{equation}
where the real-valued coefficients $a_{ij}, b_i$ are smooth for all $i, j = 1, \cdots, m$ and the matrix $a(x) = (a_{ij}(x))_{i,j=1}^m$ is a symmetric positive semidefinite quadratic form for all $x \in \bbR^m$. If $a$ is in fact (uniformly) positive definite, then we say $L$ is (uniformly) elliptic. 

In the Bakry-\'Emery language, intrinsic to $L$ is its carr\'e du champ 
\begin{equation}\label{intro:gamma}
    \Gamma(f, g) = \frac{1}{2}(L(fg) - fLg - gLf)
\end{equation}
which, at least in the setting of diffusions, encodes the second order part of $L$, in that if $L$ takes the form \eqref{intro:L}, then 
\[ \Gamma(f, g) = \sum_{i,j=1}^m a_{ij}\partial_if\partial_jg = a(\nabla_{\bbR^m} f, \nabla_{\bbR^m} g). \]
Note since $a$ is positive semidefinite the carr\'e du champ is nonnegative, which, in particular, implies the carr\'e du champ is a symmetric bilinear positive semidefinite form and therefore satisfies the Cauchy-Schwarz inequality 
\begin{equation}\label{intro:gammacs}
    \Gamma(f, g)^2 \leq \Gamma(f, f) \Gamma(g, g).
\end{equation}
We record some properties of the carr\'e du champ that appear in the sequel. The first is that the carr\'e du champ allows us to succinctly express the fact $L$ satisfies the chain rule
\begin{equation}\label{intro:chainrule1}
    L\phi(f) = \phi'(f)Lf + \phi''(f)\Gamma(f)
\end{equation}
for all smooth functions $\phi: \bbR \rightarrow \bbR$. Consequently, the carr\'e du champ, which is a derivation in both arguments, also satisfies the chain rule
\begin{equation}\label{intro:chainrule2}
    \Gamma(\phi(f), g) = \phi'(f)\Gamma(f, g).
\end{equation}
The second is that the carr\'e du champ is related to $L$ through an integration by parts formula. For this we assume the existence of a measure $\mu$ on $\bbR^m$ which is invariant and reversible for $L$, meaning respectively 
\begin{equation}\label{intro:invariance1}
    \int Lf d\mu = 0
\end{equation}
and
\begin{equation}\label{intro:symmetry}
    \int fLg d\mu = \int gLf d\mu.
\end{equation}
The integration by parts formula now follows from directly integrating the definition \eqref{intro:gamma} of the carr\'e du champ:
\begin{equation}\label{intro:ibp}
    \int fLg d\mu = -\int \Gamma(f, g)d\mu.
\end{equation}

The reader may notice we have made no mention of the function space on which the previous identities hold. One of the technical difficulties we will encounter is that the Bakry-\'Emery setting requires certain regularity assumptions which will not be a priori available. We will provide an abstract framework for which the above identities hold together with additional technical conditions under which formal Bakry-\'Emery type computations can be justified. It turns out that even in their absence we will be able to prove the Hardy inequality at the cost of being unable to pass through the multiplicative subcommutation property. In either case, we will provide examples. 

From the point of view of Hardy inequalities, technical conditions beyond functions being compactly supported away from singularities are typically unimportant, and even then for nice enough weights locally integrable at their singularities (e.g. $\abs{x}^{-2}$ in dimensions $m \geq 3$), the distinction between $f \in C_c^\infty(\bbR^m \setminus \{0\})$ and $f \in C_c^\infty(\bbR^m)$ can be dispensed with a standard density argument. However, at the level of the correspondence between a Hardy inequality and a property of a Markov semigroup, we need a function space in which formal computations can be performed. Although some standard and weighted Hardy inequalities can be achieved through some ``nice" diffusions for which the technical conditions can be explicitly verified, we may also formally realise weighted Hardy inequalities by modifying the diffusion in such a way the (possibly nonsmooth) weight is carried either by the carr\'e du champ or by the invariant measure; the technical conditions are then harder to verify. 

Before we close this chapter, we remark that although we will be interested in applying semigroup theoretic techniques with the goal being to prove Hardy inequalities, the study of (symmetric) Markov semigroups often intersects other fields of research. For instance, under some mild conditions, a Markov semigroup expresses the law of a Markov process which itself is the solution of a stochastic differential equation, so that properties in one setting transfer to another. There is also a one-to-one correspondence between a semigroup and its generator and a mechanism to pass from one to the other, so that properties of the generator can be expressed in terms of properties of the semigroup and vice versa. More relevant to this paper is the connection between a semigroup and a functional inequality. 

The discussion in this chapter has been distilled from the following. We refer the reader to \cite{Led00, Bak04, Wan06, Bau14, BGL14} for more details on Markov semigroups, and to \cite{Dav89, ENB00} for more functional analytic aspects. 

%% file: chapters/2-theorems.tex
\section{Main results}

\subsection{Multiplicative subcommutation}
In the sequel, we shall assume we are given a diffusion Markov triple $(X, \mu, \Gamma)$ as defined in \cite[\S3.4.1]{BGL14}, meaning we have a measure space $(X, \cF, \mu)$, a class $\cA_0$ of real-valued measurable functions on $X$, and a carr\'e du champ $\Gamma: \cA_0 \times \cA_0 \rightarrow \cA_0$. In doing so, the discussion presented in the previous chapter actually goes in the opposite direction, in that the assumptions of a diffusion Markov triple provide not only for $f, g \in \cA_0$ the Cauchy-Schwarz inequality \eqref{intro:gammacs} and chain rule \eqref{intro:chainrule2} for $\Gamma$, but also implicity \emph{defines} the generator $L$ of the semigroup through the integration by parts formula \eqref{intro:ibp}. The chain rule \eqref{intro:chainrule1} for $L$ follows from the chain rule for $\Gamma$, invariance \eqref{intro:invariance1} with respect to $\mu$ becomes instead an assumption, and reversibility \eqref{intro:symmetry} follows from \eqref{intro:ibp}. 

As we will always be working with a diffusion $L$ and a multiplier $W$, we will dispense of the abstraction and choose $\cA_0 = C_c^\infty(X')$ where
\[ X' = \{x \in X \mid L, W, W^{-1} \text{ are smooth and } W, W^{-1} > 0 \text{ at } x\}. \]
In doing so, we ensure that \eqref{intro:ibp} extends to $f \in C_c^\infty(X')$ and $g \in C^\infty(X')$, which more concretely justifies formal integration by parts when at least one function $f$ is compactly supported away from the singularities. 

We will also dispense of two assumptions, namely the second and seventh assumptions of a diffusion Markov triple that $\cA_0$ be dense in every $L^p(\mu)$, $1 \leq p < \infty$, and that the domain $\cD(L)$ of $L$ is the closure of the form norm
\[ Q(f, g) = \int fgd\mu + \int \Gamma(f, g)d\mu \] 
in $L^2(\mu)$, and allow our operator to instead be closed in any closed subspace $\cH \sbs L^2(\mu)$ with domain $\cD(L)$. The main reason is to allow for the possibility that the semigroup is generated by another self-adjoint extension of $L$. We define the semigroup $(P_t)_{t \geq 0}$ on $\cH$ generated by $L$ by the spectral theorem. 

Although the Bakry-\'Emery setting typically deals with Markov semigroups, meaning the semigroup preserves nonnegative functions and preserves the unit function, i.e. $P_t\mathbbm{1} = \mathbbm{1}$, the above semigroup is a priori only sub-Markov, meaning $P_tf \leq 1$ if $f \leq 1$, since, for instance, the unit function need not belong to $\cH$ which occurs in the generic case $\cH = L^2(\bbR^m)$. In particular, the invariance \eqref{intro:invariance1} of $\mu$ valid on $\cA_0$ is only true up to upper bound on $\cD(L)$, i.e. $\int Lfd\mu \leq 0$ for $f \in \cD(L)$, and $P_t$ does not increase mass on $\cH$, that is
\begin{equation}\label{intro:invariance2}
\int P_tfd\mu \leq \int fd\mu, \quad f \in \cH. 
\end{equation}
When $P_t$ is actually a Markov semigroup and $P_t\mathbbm{1} = \mathbbm{1}$, then \eqref{intro:invariance1} is valid on $\cD(L)$ and $P_t$ is mass-preserving on $\cH$ in the sense \eqref{intro:invariance2} is an equality. The ``symmetric" descriptor asserts the fact $P_t$ is symmetric in $\cH$, a property it inherits from the fact $L$ is also self-adjoint in $\cH$. Lastly, it can also be shown $(P_t)_{t \geq 0}$ is a strongly continuous contraction semigroup; this is not important here, but is important for reasons beyond the scope of this review. 

The main technical assumption we will make to ensure the proof of the pointwise subcommutation goes through is the following. Let $W \in C^\infty(X)$ where $C^\infty(X)$ consists of functions which are smooth almost everywhere. We will request that $f$ belongs to 
\begin{equation}\label{technical}
    \cA_0(W) = \{f \in \cA_0 \mid LP_s(W^2(P_tf)^2) = P_sL(W^2(P_tf)^2) \Mfor s, t > 0\}. 
\end{equation}
This appears at first a very innocuous assumption on $f$, since if $L$ generates $P_t$ then they certainly commute on $\cD(L) \sbs \cH$. However, if $L$ is a differential operator and $f$ is smooth, then $Lf$ a priori has meaning even if $f \notin \cD(L)$ and we may consider the question of whether $L$ and $P_t$ commute on a larger class of functions than $\cD(L)$ which, as observed in \cite[pp.~138]{BGL14}, is essentially a matter of whether we may integrate by parts. For instance, if $L = \Delta_{\bbR^m}$ and $(P_t)_{t \geq 0}$ is the standard euclidean heat semigroup, then it can be shown by explicit computation that $\partial_t P_tf = LP_tf = P_tLf$ holds for $f$ the product of a polynomial and an exponential.

\begin{theorem}[{\textup{\cite[Theorem~3]{RZ21}}}]\label{thm:multineq}
Let $0 \leq W \in C^\infty(X)$ and $f \in \cA_0(W)$. The pointwise subcommutation 
\begin{equation}\label{multineq} 
    \Phi^W(P_tf) = W^2(P_tf)^2 \leq e^{-2\gamma t} P_t(W^2f^2) = e^{-2\gamma t} P_t\Phi^W(f)
\end{equation}
holds for some constant $\gamma \in \bbR$ and $t \geq 0$ if and only if the curvature condition
\begin{equation}\label{truecond}
\begin{aligned}
    \Gamma^W(f) = \frac{1}{2}(L(W^2)f^2 + 2\Gamma(W^2, f^2) + 2W^2\Gamma(f)) \geq \gamma \Phi^W(f)
\end{aligned}
\end{equation}
holds, and the curvature condition \eqref{truecond} holds if
\begin{equation}\label{suffcond}
    \inf_{W \neq 0} \frac{LW}{W} - 3 \frac{\Gamma(W)}{W^2} \geq \gamma > -\infty.
\end{equation}
\end{theorem}

\begin{proof}
    The proof of the curvature condition and the sufficient condition was originally done for smooth euclidean Kolmogorov type operators and smooth $W$ in \cite{RZ21}, but a formal generalisation to arbitrary diffusions is straightforward; the main ingredients of the proof are the respective chain rules \eqref{intro:chainrule1} and \eqref{intro:chainrule2} for $L$ and $\Gamma$, and the Cauchy-Schwarz inequality \eqref{intro:gammacs} for $\Gamma$.

    The formal computations can be justified by assumption of \eqref{technical}; the proof involves differentiating $P_s(W^2(P_{t-s}f)^2)$ with respect to $s$, which yields 
    \[ \partial_sP_s(W^2(P_{t-s}f)^2) = LP_s(W^2(P_{t-s}f)^2) + P_s(\partial_s(W^2(P_{t-s}f)^2)) \] 
    and since $f \in \cA_0(W)$, it follows $L$ and $P_s$ commute. 
\end{proof}

Our next objective is the correspondence between the integrated version of the pointwise subcommutation \eqref{multineq}, that is the integrated $\Phi^W$-subcommutation, and the functional inequality appearing in the introduction. The main point is that, as usual in semigroup theory, differentiation at time $t = 0$ yields an equivalence with a functional inequality. 

\begin{theorem}\label{thm:intineq}
Let $0 \leq W \in C^\infty(X)$ and $f \in \cA_0(W)$. The integrated subcommutation
\begin{equation}\label{intineq}
    \int W^2(P_tf)^2 d\mu \leq e^{-2\gamma t} \int W^2f^2 d\mu
\end{equation}
holds if and only if
\begin{equation}\label{funcineq0}
    \int L(W^2)f^2d\mu \leq 2\int W^2\Gamma(f)d\mu - 2\gamma \int W^2f^2d\mu. 
\end{equation}
That is, the functional inequality \eqref{funcineq0} is equivalent to the weighted contractivity condition
\begin{equation}\label{contraction}
    \norm{P_tf}_{L^2(W^2d\mu)} = \norm{W P_tf}_{L^2(\mu)} \leq e^{-\gamma t} \norm{W f}_{L^2(\mu)} = e^{-\gamma t}\norm{f}_{L^2(W^2d\mu)}.
\end{equation}
\end{theorem}

\begin{remark}
If $\gamma \neq 0$ then this property can be regarded as a weighted contraction property for the (nondiffusive) semigroup generated by $L + \gamma$.
\end{remark}

\begin{proof}
Define $I(t) = \int W^2(P_tf)^2d\mu$. With $g = P_tf$, 
\begin{align*} 
    I'(t) &= 2\int W^2gLgd\mu \\
    &= \int W^2L(g^2)d\mu - 2\int W^2\Gamma(g)d\mu \\
    &= \int L(W^2)g^2d\mu - 2\int W^2\Gamma(g)d\mu
\end{align*}
by the chain rule \eqref{intro:chainrule1} and reversibility \eqref{intro:symmetry}. If \eqref{intineq} holds then $I'(t) \leq -2\gamma I(t)$ which is \eqref{funcineq0} at $t = 0$, and conversely if \eqref{funcineq0} holds then $I'(t) \leq -2\gamma I(t)$ yields by Gronwall's lemma $I(t) \leq e^{-2\gamma t}I(0)$ which is \eqref{intineq}. 
\end{proof}

The connection between \eqref{multineq} and \eqref{intineq} is that \eqref{multineq} is clearly sufficient, but perhaps not necessary, for \eqref{intineq}. Thus as mentioned earlier in the introduction, the sufficient condition \eqref{suffcond} sequentially verifies the curvature condition \eqref{truecond}, the pointwise subcommutation \eqref{multineq}, the integrated subcommutation \eqref{intineq}, and finally the functional inequality \eqref{funcineq0}.

\subsection{The abstract Hardy inequality}
We shall now specialise to the case $\gamma = 0$ and show how the functional inequality \eqref{funcineq0} gives rise to a Hardy inequality. Note the form of \eqref{funcineq0} already has the ``correct" order in the sense $L(W^2)$ is two orders lower than $W^2$. We start with two lemmata. The first is a simple transformation of \eqref{funcineq0} and the second is a sufficient geometric condition to verify the first. 

\begin{lemma}\label{lem:funcineq}
Let $f \in \cA_0(W)$ and $0 \leq W \in C^\infty(X)$ verify \eqref{multineq} with constant $\gamma = 0$. Then 
\begin{equation}\label{funcineq}
    \int \frac{\Gamma(W)}{W^2}f^2d\mu \leq \int \Gamma(f)d\mu
\end{equation}
and more generally if $\beta \in \bbR$ then
\begin{equation}\label{funcineqgeneral}
    (1-\beta)\int W^{1-2\beta}LW f^2d\mu + (\beta^2-\beta+1)\int W^{-2\beta}\Gamma(W)f^2d\mu \leq \int W^{2-2\beta}\Gamma(f)d\mu.     
\end{equation}
\end{lemma}

\begin{proof}
For the first part, if $W$ verifies \eqref{multineq}, then it verifies \eqref{intineq} with the same constant by integrating with respect to $\mu$ and therefore 
\[ \int L(W^2)f^2d\mu \leq 2\int W^2\Gamma(f)d\mu - 2\gamma\int W^2f^2d\mu \]
by Theorem \ref{thm:intineq}. Replacing $f$ with $f/W$, and integrating by parts the cross term in the expansion of $\Gamma(f/W) = W^{-2}\Gamma(f) + 2W^{-1}f\Gamma(W^{-1}, f) + \Gamma(W^{-1})f^2$ we find 
\begin{align*}
    2\int W^2 \cdot W^{-1}\Gamma(W^{-1}, f)d\mu = -\int\Gamma(\log W, f^2)d\mu &= \int L(\log W)f^2d\mu \vphantom{\int \left(\frac{LW}{W} - \frac{\Gamma W}{W^2}\right)f^2d\mu} \\
    &= \int \left(\frac{LW}{W} - \frac{\Gamma(W)}{W^2}\right)f^2d\mu. 
\end{align*}
Hence 
\begin{align*}
    2\int \left(\frac{LW}{W} + \frac{\Gamma(W)}{W^2}\right)f^2d\mu &\leq 2\int \Gamma(f)d\mu + 2\int W^2\Gamma(W^{-1})f^2d\mu \\
    &+ 2\int \left(\frac{LW}{W} - \frac{\Gamma(W)}{W^2}\right)f^2d\mu.
\end{align*}
Since $W^2\Gamma(W^{-1}) = W^{-2}\Gamma(W)$, simplifying gives the result. 

For the second part, replacing $f$ with $f/W^\beta$ for some $\beta \neq 1$ instead, the integration by parts gives 
\[ \int \left(\frac{L(W^2)}{W^{2\beta}} + \frac{2\beta}{2\beta-2} L(W^{2-2\beta}) - 2\beta^2 \frac{\Gamma(W)}{W^{2\beta}}\right)f^2d\mu \leq 2 \int W^{2-2\beta}\Gamma(f)d\mu \]
which simplifies to \eqref{funcineqgeneral}. 
\end{proof}

\begin{lemma}\label{lem:suffcond}
Let $0 \leq \psi \in C^\infty(X)$ verify
\begin{equation}\label{qcond} 
    L\psi = \frac{Q - 1}{\psi}\Gamma(\psi)
\end{equation}
for some $Q \neq 2$. Then $W = \psi^p$ verifies \eqref{multineq} with constant $\gamma = 0$ whenever $p$ is between $0$ and $\frac{Q - 2}{2} \in \bbR_{\neq 0}$.
\end{lemma}

\begin{proof}
It suffices to show $W = \psi^p$ verifies \eqref{suffcond} by Theorem \ref{thm:multineq}. Note since $L$ is a differential operator and $\psi$ is smooth almost everywhere, $L\psi$ is defined almost everywhere. By \eqref{intro:chainrule1} and \eqref{intro:chainrule2}, 
\[ \frac{LW}{W} - 3\frac{\Gamma(W)}{W^2} = (p(Q - 1) + p(p - 1) - 3p^2)\frac{\Gamma(\psi)}{\psi^2} \]
and clearly $p(p + Q - 2) - 3p^2$ is nonnegative whenever $p$ is between $0$ and $\frac{Q - 2}{2}$.
\end{proof}

\begin{remark}
If $Q = 2$ then clearly $p(p + Q - 2) - 3p^2 = -2p^2$ is nonnegative if and only if $p = 0$, but of course this recovers the trivial multiplier $W \equiv 1$. It may well be the case $\Gamma(\psi)\psi^{-2}$ is bounded, in which case $W$ verifies \eqref{suffcond} with constant $\gamma = -2p^2\sup\Gamma(\psi)\psi^{-2}$, but in this paper we are typically interested in the case $\gamma = 0$ and with choices of $\psi$ for which it will typically be the case $\Gamma(\psi)\psi^{-2}$ is unbounded.
\end{remark}

Our next result concerns the Hardy inequality in both its usual high-dimensional form and its critical $2$-dimensional form with logarithmic correction. 

\begin{theorem}\label{mainthm}
Let $0 \leq \psi \in C^\infty(X)$ verify \eqref{qcond} for some $Q \in \bbR$. If $Q \neq 2$ then the Hardy inequalities
\begin{equation}\label{hardy}
    \int \frac{\Gamma(\psi)}{\psi^2}f^2d\mu \leq \left(\frac{Q-2}{2}\right)^2 \int \Gamma(f)d\mu
\end{equation}
and more generally for $Q + \alpha \neq 2$
\begin{equation}\label{hardygeneral}
    \int \psi^\alpha \frac{\Gamma(\psi)}{\psi^2}f^2d\mu \leq \left(\frac{2}{Q + \alpha - 2}\right)^2 \int \psi^\alpha \Gamma(f)d\mu
\end{equation}
holds for all $f \in \cA_0(\psi^{(Q-2)/2})$. If $Q = 2$, either $\psi \leq 1$ or $\psi \geq 1$, and $\log^{-1/2}\psi \in C^\infty(X)$ then the Hardy inequalities
\begin{equation}\label{loghardy}
    \int \frac{\Gamma(\psi)}{\psi^2\log^2\psi}f^2d\mu \leq 4 \int \Gamma(f)d\mu
\end{equation}
and more generally for $\alpha \neq 1$
\begin{equation}\label{loghardygeneral}
    \int \log^\alpha \psi \frac{\Gamma(\psi)}{\psi^2 \log^2\psi} f^2d\mu \leq \left(\frac{2}{\alpha-1}\right)^2 \int \log^\alpha \psi \Gamma(f)d\mu
\end{equation}
holds for all $f \in \cA_0(\log^{-1/2}\psi)$ and where $\log^p\psi \vcentcolon= \abs{\log \psi}^p$. 
\end{theorem}

The requirement that $f$ be supported away from the singularities of the weight on the left hand side, implicit in the definition of $\cA_0(W) \sbs \cA_0 = C_c^\infty(X') \sbs C_c^\infty(X)$, can be removed, for instance if $\Gamma(\psi)\psi^{-2}$ is locally integrable near its singularities. However, since we are interested in Hardy inequalities with various weights on either side, to simplify the exposition we give $f$ enough regularity to sidestep such considerations. Note as usual the inequalities above hold for $f$ in a larger space by standard density arguments. 

To motivate the expression \eqref{hardy}, the prototypical example is that of $L = \Delta$ the euclidean laplacian on $\bbR^m$ whose carr\'e du champ $\Gamma(f) = \abs{\nabla f}^2$ is the square of the euclidean gradient and reversible measure $d\mu$ is Lebesgue measure $dx$ on $\bbR^m$. If $\psi = \abs{x}$ is the euclidean norm and $m \geq 3$, then $\psi$ verifies \eqref{qcond} with constant $Q = m$, and \eqref{hardy} is then precisely the sharp Hardy inequality on $\bbR^m$ while \eqref{loghardy} is the sharp Hardy inequality with logarithmic correction on $\bbR^2$. In general, $Q$ represents a ``homogeneous dimension" of sorts which we elucidate in the examples to follow. Note although the constant $(\frac{2}{Q - 2})^2$ is not optimal in general, it turns out that it is optimal in a number of generic cases. (This may be somewhat surprising given the starting point is a sufficient condition for weighted contractivity.)

\begin{proof} 
If $Q \neq 2$, by Lemma \ref{lem:suffcond} the multiplier $W = \psi^{(Q-2)/2} \in C^\infty(X)$ verifies \eqref{multineq} with constant $\gamma = 0$, and by the first part of Lemma \ref{lem:funcineq} it holds
\[ \int \frac{\Gamma(\psi)}{\psi^2}f^2d\mu \leq \left(\frac{2}{Q - 2}\right)^2 \int \Gamma(f)d\mu \]
which is \eqref{hardy}. If $Q = 2$ define $\Phi = \log \psi$. Then 
\[ L\Phi = \frac{1}{\psi}L\psi - \frac{\Gamma(\psi)}{\psi^2} = 0, \]
i.e. $\Phi$ verifies \eqref{qcond} with constant $Q = 1$. Since $W = (\log \psi)^{(Q-2)/2} = (\log \psi)^{-1/2} \in C^\infty(X)$ and noting $\Gamma(\Phi) = \Gamma(\psi)\psi^{-2}$ we see that \eqref{loghardy} follows from \eqref{hardy}. Similarly, the second part of Lemma \ref{lem:funcineq} gives \eqref{hardygeneral} and \eqref{loghardygeneral} for $\beta$ satisfying $\alpha = (Q - 2)(1 - \beta)$ and $\alpha = \beta - 1$ respectively.
\end{proof}

\begin{remark}
Note since $\Phi < 0$ in the region $\{\psi \leq 1\}$, the case $Q = 2$ is more delicate than the previous proof suggests. We leave this discussion to the final section of the chapter. 
\end{remark}

One advantage of proving the Hardy inequality in this way is that the arguments are easily adapted to the case where only a one-sided laplacian comparison holds, that is when only an upper bound or lower bound on $L\psi$ is given. In particular, the former occurs naturally in the setting of riemannian manifolds $(M, g)$ with nonnegative Ricci curvature as a consequence of the celebrated laplacian comparison theorem. 

\begin{corollary}\label{cor:lower}
Let $0 \leq \psi \in C^\infty(X)$ verify \eqref{qcond} from below, i.e. $\psi$ verifies
\begin{equation}\label{qcondlower}
    L\psi \geq \frac{Q - 1}{\psi}\Gamma(\psi)
\end{equation}
for $Q > 2$. Then \eqref{hardygeneral} holds for any $\alpha \geq 0$.
\end{corollary}

\begin{corollary}\label{cor:upper}
Let $0 \leq \psi \in C^\infty(X)$ verify \eqref{qcond} from above, i.e. $\psi$ verifies
\begin{equation}\label{qcondupper}
    L\psi \leq \frac{Q - 1}{\psi}\Gamma(\psi)
\end{equation}
for some $Q < 2$. Then \eqref{hardygeneral} holds for any $\alpha \leq 0$.
\end{corollary}

Note the above corollaries are silent on the case $Q = 2$. In fact, the logarithmic Hardy inequality \eqref{loghardy} can be recovered in this case, but we defer its proof to the next section. 

\begin{proof}
We assume we are given a lower bound, the upper bound being similar. In verifying \eqref{suffcond}, note a lower bound \eqref{qcondlower} on $L\psi$ furnishes a lower bound on $L\psi^p$ whenever $p > 0$. But since $p = \frac{Q - 2}{2}$, this implies $Q > 2$. In fact, $p$ must be between $0$ and $\frac{Q - 2}{2}$ as far as replicating the proof of Theorem \ref{mainthm} is concerned, so $Q > 2$ is necessary. This recovers \eqref{funcineq} and \eqref{funcineqgeneral} by Lemma \ref{lem:funcineq}, since if $\beta < 1$, in which case $\alpha = (Q-2)(1-\beta) \geq 0$, then the left hand side of \eqref{funcineqgeneral} can be bounded below; the rest of the proof follows like in the equality case. 
\end{proof}

Returning to the case of equality, let us emphasise that the more important fact from the viewpoint of semigroup theory is that the Hardy inequalities here follow from the weighted contraction property \eqref{contraction}. Reframed to emphasise the connection with semigroups, the theorems given in this section may be reformulated as follows. 

\begin{theorem}\label{thm:contraction0}
Let $0 \leq \psi \in C^\infty(X)$ verify \eqref{qcond} with constant $Q \neq 2$. Then the semigroup $(P_t)_{t \geq 0}$ satisfies the weighted contraction property 
\[ \norm{\psi^{(Q - 2)/2}P_tf}_{L^2(\mu)} \leq \norm{\psi^{(Q - 2)/2}f}_{L^2(\mu)} \]
if and only if
\[ \int \psi^{Q-2}\frac{\Gamma(\psi)}{\psi^2}f^2d\mu \leq \left(\frac{1}{Q-2}\right)^2 \int \psi^{Q-2}\Gamma(f)d\mu \]
if and only if for $Q + \alpha \neq 2$
\[ \int \psi^\alpha \frac{\Gamma(\psi)}{\psi^2}f^2d\mu \leq \left(\frac{2}{Q + \alpha - 2}\right)^2 \int \psi^\alpha \Gamma(f)d\mu \] 
holds for all $f \in \cA_0(\psi^{(Q-2)/2})$. 
\end{theorem}

\begin{proof}
The first equivalence follows from Theorem \ref{thm:intineq}, and the second equivalence follows from the observation the computation performed in Lemma \ref{lem:funcineq} is invertible, i.e. \eqref{funcineq0} and \eqref{funcineq} are equivalent and related by the transformation $f \rightarrow f/W^\beta$ for any $\beta \in \bbR$.
\end{proof}

\subsection{The logarithmic $Q = 2$ case}
As mentioned earlier, the case $Q = 2$ is more delicate than the proof of Theorem \ref{mainthm} would make it seem, since $\Phi = \log \psi$ is negative if $\psi \leq 1$. Since the function $\psi$ is replaced with $\Phi = \log \psi$ which verifies \eqref{qcond} with constant $Q = 1$, it makes no sense to speak of $W = \Phi^{(Q - 2)/2} = (\log \psi)^{-1/2}$ in this region.

What makes the theory consistent in this case is the observation that \eqref{qcond} is invariant with respect to nonzero scalar multiplication, that is it holds for $\Phi$ if and only if it holds for $\lambda \Phi$ for any $\lambda \neq 0$. We proceed as follows. There are two cases: either $\psi \leq 1$ or $\psi \geq 1$. If the former, we replace $\psi$ with $-\Phi$, and if the latter, we replace $\psi$ with $\Phi$. The Hardy inequality in both cases reads exactly the same so this distinction makes no difference to the logarithmic Hardy inequality \eqref{loghardygeneral}, but nonetheless this distinction is important since, for instance, it affects the statements of Theorem \ref{thm:contraction0} and Corollaries \ref{cor:lower}, \ref{cor:upper} from the point of view of the semigroup in the logarithmic case. 

\begin{theorem}
    Let $0 \leq \psi \in C^\infty(X)$ verify \eqref{qcond} with constant $Q = 2$. If either $\psi \leq 1$ and $\Phi = -\log\psi$ or $\psi \geq 1$ and $\Phi = \log\psi$, and $\Phi^{-1/2} \in C^\infty(X)$, then the semigroup $(P_t)_{t \geq 0}$ satisfies the weighted contraction property 
    \[ \norm{\Phi^{-1/2}P_tf}_{L^2(\mu)} \leq \norm{\Phi^{-1/2}f}_{L^2(\mu)} \]
    if and only if 
    \[ \int \frac{\Gamma(\psi)}{\psi^2\log^2\psi}f^2d\mu \leq 4 \int \Gamma(f)d\mu \]
    if and only if for $\alpha \neq 1$
    \[ \int \log^\alpha\psi \frac{\Gamma(\psi)}{\psi^2\log^{2}\psi}f^2d\mu \leq \left(\frac{2}{\alpha-1}\right)^2 \int \log^\alpha\psi \Gamma(f)d\mu \]
    holds for all $f \in \cA_0(\Phi^{-1/2})$. 
\end{theorem}

\begin{proof}
    The proof follows the proof of Theorem \ref{thm:contraction0} with $W = (\pm \log \psi)^{-1/2}$.
\end{proof}

\begin{corollary}\label{cor:loglower}
Let $0 \leq \psi \in C^\infty(X)$ satisfy $\psi \leq 1$ and verify \eqref{qcondlower} for $Q = 2$, and $\Phi = (-\log \psi)^{-1/2} \in C^\infty(X)$. Then \eqref{loghardygeneral} holds for $\alpha \geq 0$.
\end{corollary}

\begin{corollary}\label{cor:logupper}
Let $0 \leq \psi \in C^\infty(X)$ satisfy $\psi \geq 1$ and verify \eqref{qcondupper} for $Q = 2$, and $\Phi = (\log \psi)^{-1/2} \in C^\infty(X)$. Then \eqref{loghardygeneral} holds for $\alpha \leq 0$.
\end{corollary}

\begin{proof}
We assume we are given a lower bound, the upper bound being similar. If $\psi$ verifies the laplacian lower bound \eqref{qcondlower} with constant $Q = 2$ then $L(\log \psi) \geq 0$, that is $\log \psi$ verifies \eqref{qcondlower} with constant $Q = 1$. But for a one-sided laplacian comparison with constant $Q = 1 < 2$ to yield a Hardy inequality in this setup, we need a laplacian upper bound \eqref{qcondupper}. Replacing $\Phi$ with $-\Phi$, the result now follows from Lemma \ref{lem:funcineq}, since if $\beta > 1$, in which case $\alpha = \beta - 1 \geq 0$, then the left hand side of \eqref{funcineqgeneral} can be bounded below. 
\end{proof}

In practice, whether or not $\psi \leq 1$ or $\psi \geq 1$ is a matter of whether the underlying space is a subset of $\{\psi \leq 1\}$ or $\{\psi \geq 1\}$; from the point of view of the semigroup these two spaces give rise to two different semigroups, say, in the case of the Dirichlet laplacian, two Brownian motions killed at the common boundary $\{\psi = 1\}$.

\subsection{Absence of the technical assumption}\label{s3.4}
Before we close this chapter, let us return to the statement made in the previous chapter that the Hardy inequality can still be obtained even in the absence of the technical assumption can be sidestepped entirely. 

Indeed, the Hardy inequalities themselves can be obtained without going through the semigroup, because the pointwise inequality 
\begin{align*}
    L(W^2)f^2 + 2\Gamma(W^2, f^2) + 2W^2\Gamma(f) &\geq (2WLW - 6\Gamma(W))f^2 \vphantom{2\inf_{W \neq 0} \left(\frac{LW}{W} - 3\frac{\Gamma(W)}{W^2}\right)W^2f^2} \\
    &\geq 2\inf_{W \neq 0} \left(\frac{LW}{W} - 3\frac{\Gamma(W)}{W^2}\right)W^2f^2 
\end{align*}
for $f$ compactly supported away from the singularities of $L$ and $W$ can be directly integrated with respect to $\mu$ to obtain, after integrating by parts, 
\[ -\int L(W^2)f^2d\mu + 2\int W^2\Gamma(f) \geq 2\gamma \int W^2f^2d\mu, \] 
the starting point of Lemma \ref{lem:funcineq}. Thus in what is to follow, all results hold, even if passage through the semigroup is not necessarily true.

In practice, the condition appears difficult to verify except in particular cases in which one is provided explicit knowledge of the mapping properties of the semigroup (perhaps together with explicit heat kernel estimates). For instance, although the discussion of this chapter can be extended at least formally to generators with weighted carr\'e du champs and reversible measures (also called weighted or drifted laplacians respectively in the literature), we have intentionally avoided their discussion since the technical assumption in such cases are even harder to verify. Nonetheless, for completeness' sake, let us remark that the weighted Hardy inequality \eqref{hardygeneral} can be realised as a standard Hardy inequality \eqref{hardy} with respect to a suitable weight. Indeed, If $0 \leq \omega \in C^1(X)$ then the operator defined by 
\[ L_\omega f = \omega Lf + \Gamma(\omega, f) \]
is a diffusion with the same reversible measure $\mu_\omega = \mu$ as $L$ but with a weighted carr\'e du champ $\Gamma_\omega(f, g) = \omega \Gamma(f, g)$. Thus assuming $\psi$ verifies \eqref{qcond} with constant $Q \in \bbR$ for $L$, by taking $\omega = \psi^\alpha$ we find
\begin{align*}
    L_\omega \psi = \psi^\alpha L \psi + \Gamma(\psi^\alpha, \psi) &= \psi^\alpha \frac{Q - 1}{\psi} \Gamma(\psi) + \alpha \psi^{\alpha - 1} \Gamma(\psi) \\
    &= \frac{Q + \alpha - 1}{\psi} \psi^\alpha \Gamma(\psi) \\
    &= \frac{Q + \alpha - 1}{\psi} \Gamma_\omega(\psi, \psi),
\end{align*}
i.e. $\psi$ verifies \eqref{qcond} with constant $Q + \alpha$ for $L_\omega$ and the result follows from Theorem \ref{mainthm}. Alternatively, the proof can be done so that the factor $\psi^\alpha$ is carried by $\mu$ instead of $\Gamma$. If $\sigma \in C^1(X)$ then the operator defined by $L_\sigma f = Lf + \Gamma(\sigma, f)$ is a diffusion with the same carr\'e du champ $\Gamma_\sigma = \Gamma$ as $L$ but with a weighted reversible measure $d\mu_\sigma = e^\sigma d\mu$. By taking $\sigma = \log \psi^\alpha = \alpha \log \psi$ we find
\[ L_\sigma\psi = L\psi + \Gamma(\sigma, \psi) = \frac{Q - 1}{\psi}\Gamma(\psi) + \frac{\alpha}{\psi}\Gamma(\psi) = \frac{Q + \alpha - 1}{\psi}\Gamma_\sigma(\psi), \]
i.e. $\psi$ verifies \eqref{qcond} with constant $Q + \alpha$ for $L_\sigma$ as expected. In either case, at least formally, the Hardy inequality can be interpreted as a weighted contraction property for the semigroups generated by $L_\omega$ and $L_\sigma$. This leads to new inequalities in the logarithmic setting, by choosing $\alpha$ such that $Q + \alpha = 2$.  

\begin{theorem}
    Let $0 \leq \psi \in C^\infty(X)$ verify \eqref{qcond} with $Q \neq 2$, either $\psi \leq 1$ and $\Phi = -\log\psi$ or $\psi \geq 1$ and $\Phi = \log\psi$, and $\Phi^{-1/2} \in C^\infty(X)$, then the Hardy inequalities 
    \[ \int \psi^{2-Q} \frac{\Gamma(\psi)}{\psi^2\log^2\psi} f^2d\mu \leq 4 \int \psi^{2-Q} \Gamma(f)d\mu \]
    and more generally for $\alpha \neq 1$
    \[ \int \psi^{2-Q} \log^\alpha\psi \frac{\Gamma(\psi)}{\psi^2\log^{2}\psi} f^2d\mu \leq \left(\frac{2}{\alpha-1}\right)^2 \int \psi^{2-Q} \log^\alpha\psi \Gamma(f)d\mu \] 
    hold for $f \in \cA_0(\Phi^{-1/2})$. 
\end{theorem}

\begin{remark}
    This method of proving Hardy inequalities, namely integrating a pointwise inequality, might be contrasted with the method of ``expanding the square" which integrates the pointwise inequality $\abs{v}^2 \geq 0$ for any vector $v$. In particular, the choice $v = \nabla f + \alpha \frac{x}{\abs{x}^2}f$ integrated against Lebesgue measure recovers the euclidean Hardy inequality on $\bbR^m$ with sharp constant after integrating by parts and maximising the constant in $\alpha$. 
\end{remark}

%% file: chapters/3-applications.tex
\section{Applications}
There is a large literature of Hardy inequalities in various settings. Here we document a few Hardy inequalities appearing in the existing literature by giving a diffusion $L$ together with a weight $\psi$ verifying \eqref{qcond} for some constant $Q \in \bbR$. We will present results without passing through the technical assumption \eqref{technical}. In view of \hyperref[s3.4]{\S3.4}, this means that the results hold irrespective of whether or not a pointwise subcommutation \eqref{multineq} for a semigroup is provided. 

\subsection{Stratified Lie groups}\label{sec4.1}
\addtocontents{toc}{\protect\setcounter{tocdepth}{0}}
A Lie group $\bbG = (\bbR^m, \circ)$ is said to be stratified if its Lie algebra $\mf{g}$ admits a stratification $\mf{g} = \otimes_{i=0}^{r-1} \mf{g}_i$ where $\mf{g}_i = [\mf{g}_0, \mf{g}_{i-1}]$ for $1 \leq i < r-1$. In this setting, there are at least two Hardy type inequalities of interest. 

Our first result is a ``horizontal" Hardy inequality. The stratification of $\mf{g}$ induces a coordinate system on $\bbG \cong \otimes_{i=0}^{r-1} \bbR_i$ where $\bbR_i \cong \bbR^{\dim(\mf{g}_i)}$ and $\dim(\mf{g}_i)$ is the topological dimension of $\mf{g}_i$. The first stratum $\mf{g}_0$ is called the horizontal stratum, and it admits a canonical basis $\{X_1, \cdots, X_\ell\}$ of vector fields generating $\mf{g}$. The operator $\nabla_\bbG = (X_1, \cdots, X_\ell)$ is called the horizontal gradient or also the subgradient, the coordinates $x_0 \in \bbR_0$ associated to $\mf{g}_0$ are called the horizontal coordinates, and their euclidean norm $\abs{x_0}$ is called the horizontal norm. Let us emphasise that, in analogue with the euclidean setting, the diffusion $L$ is the sublaplacian $\Delta_\bbG = \nabla_\bbG \cdot \nabla_\bbG = \sum_{i=1}^\ell X_i^2$, its carr\'e du champ $\Gamma(f) = \abs{\nabla_\bbG f}^2$ is the square of the subgradient, and its reversible measure $\mu$ is Lebesgue measure $d\xi$ on $\bbG$. For more details on stratified Lie groups we refer the reader to \cite[Chapter~1]{BLU07}. 

\begin{corollary}
Let $\bbG$ be a stratified Lie group, $x_0' \sbs x_0 \in \bbR_0$ a subcoordinate system of dimension $n_0 \geq 2$ consisting solely of horizontal coordinates. Then 
\[ \int_{\bbG} \abs{x_0'}^\alpha\frac{f^2}{\abs{x_0'}^2}d\xi \leq \left(\frac{2}{n_0+\alpha-2}\right)^2 \int_{\bbG} \abs{x'}^\alpha\abs{\nabla_\bbG f}^2d\xi \]
holds for $f \in C_c^\infty(\bbG \setminus \{\abs{x_0'} = 0\})$ and $\alpha$ such that $n_0 + \alpha \neq 2$, and
\[ \int_{\bbG} \abs{x_0'}^{2-n_0}\log^\alpha\abs{x_0'}\frac{f^2}{\abs{x_0'}^2\log^2\abs{x_0'}}d\xi \leq \left(\frac{2}{\alpha-1}\right)^2 \int_{\bbG} \abs{x_0'}^{2-n}\log^\alpha\abs{x_0'}\abs{\nabla_\bbG f}^2d\xi \]
holds for $f \in C_c^\infty(\bbG \setminus \{\abs{x_0'} \in \{0, 1\}\})$ and $\alpha \neq 1$. 
\end{corollary}

This was first proven by \cite[Theorem~3.3]{Amb04} on the Heisenberg group $\bbH^m$, and later generalised to a stratified Lie group $\bbG$ by \cite[Corollary~5.7]{RS17}.

\begin{proof}
The subgradient coincides with the euclidean gradient for a function depending only on horizontal coordinates, so $\abs{x_0'}$ verifies \eqref{qcond} with constant $Q = n_0$ for $L = \Delta_\bbG$.
\end{proof} 

\begin{remark}
    From the Bakry-\'Emery perspective, the technical assumption \eqref{technical} can be verified for $L = \Delta_\bbG$ generating the standard semigroup $(P_t)_{t \geq 0}$ on $L^2(\bbG)$, at least in the euclidean setting. It is known that the sublaplacian $\Delta_\bbG$ is hypoelliptic and has a smooth heat kernel satisfying gaussian decay estimates. If $f$ is smooth and compactly supported, then $P_{t-s}f$ is smooth and bounded, in fact Schwartz since $L$ preserves Schwartz functions (see \cite[Lemma~2.1]{Bak08}). It follows $(P_{t-s}f)^2$ is Schwartz and therefore, at least in dimension $n_0 \geq 3$ where $W = \abs{x_0'}^{(n-2)/2}$, that $W^2(P_{t-s}f)^2 \in L^2(\bbG)$ and $L(W^2(P_{t-s}f)^2) \in L^2(\bbG)$ since $W^2$, $\Gamma(W)$, and $L(W^2)$ are all locally integrable and polynomially bounded. Since $L = \Delta_\bbG$ is essentially self-adjoint on $C_c^\infty(\bbG) \sbs L^2(\bbG)$ and its self-adjoint extension has domain $H^2(\bbG) = \{u \in L^2(\bbG) \mid Lu \in L^2(\bbG)\}$, we may then conclude $W^2(P_{t-s}f)^2 \in \cD(L)$ on which $L$ and $P_s$ commute, and \eqref{technical} follows.
    
    Alternatively, we may first approximate $W$ with the smooth $W_\epsilon = (\epsilon + \abs{x_0'}^2)^{(n_0-2)/4}$ for $\epsilon > 0$ and check the sufficient condition \eqref{suffcond}. A direct computation yields 
    \[ \frac{LW_\epsilon}{W_\epsilon} - 3\frac{\Gamma(W_\epsilon)}{W_\epsilon^2} = \frac{\epsilon n_0(n_0 - 2)}{2}(\epsilon + \abs{x_0'}^2)^{-2} \geq 0, \]
    i.e. $W_\epsilon$ verifies \eqref{suffcond} with constant $\gamma = 0$. Since $W_\epsilon^2(P_{t-s}f)^2$ is Schwartz, it belongs to $\cD(L)$, and \eqref{technical} follows. Thus $W_\epsilon^2(P_tf)^2 \leq P_t(W_\epsilon^2f^2)$ holds, and dominated convergence yields $W^2(P_tf)^2 \leq P_t(W^2f^2)$. 
    
    Even on a bounded domain $0 \ni \Omega \sbs \bbR^m$ with $C^\infty$-boundary and containing a neighbourhood of the origin, the analysis can be complicated. Consider for instance the sub-Markov semigroup $(P_t^D)_{t \geq 0}$ generated by the euclidean Dirichlet laplacian $L = \Delta_\Omega^D$. We wish to know if $W_\epsilon^2(P_{t-s}^Df)^2$ belongs to $\cD(L) = H_0^1(\Omega) \cap H^2(\Omega)$. However, only knowing $P_{t-s}^Df \in \cD(L)$ by standard semigroup theory appears inadequate, since neither $H_0^1(\Omega)$ nor $H^2(\Omega)$ are algebras in general. But from the probabilistic interpretation of the semigroup as an expectation (see \cite[pp.~139]{BGL14}), we have $LP_tf = P_tLf$ whenever $f$ and $Lf$ are continuous and bounded, which holds since $W_\epsilon^2(P_{t-s}f)^2$ is a smooth function on a bounded domain. The same holds by similar considerations for the Markov semigroup generated by the Neumann laplacian.    
\end{remark}

Our second result is a ``homogeneous" Hardy inequality. On any stratified Lie group of homogeneous dimension $Q(\bbG) = \sum_{i=0}^{r-1} (i+1)\dim(\mf{g}_i) \geq 3$, which we assume here and in the sequel, there exists by \cite[Theorem~2.1]{Fol75} a fundamental solution $u$ of $\Delta_\bbG$. The function $N = u^{1/(2-Q(\bbG))}$ is called the Kor\'anyi-Folland gauge. 

\begin{corollary}\label{cor:hardyN}
Let $\bbG$ be a stratified Lie group. Then 
\[ \int_{\bbG} N^\alpha\frac{\abs{\nabla_\bbG N}^2}{N^2}f^2d\xi \leq \left(\frac{2}{Q(\bbG) + \alpha - 2}\right)^2 \int_{\bbG} N^\alpha\abs{\nabla_\bbG f}^2d\xi \]
holds for $f \in C_c^\infty(\bbG \setminus \{0\})$ and $\alpha$ such that $Q(\bbG) + \alpha \neq 2$, and 
\[ \int_{\bbG} N^{2-Q(\bbG)}\log^\alpha N\frac{\abs{\nabla_\bbG N}^2}{N^2\log^2N}f^2d\xi \leq \left(\frac{2}{\alpha-1}\right)^2 \int_{\bbG} N^{2-Q(\bbG)}\log^\alpha N\abs{\nabla_\bbG f}^2d\xi \] 
holds for $f \in C_c^\infty(\bbG \setminus \{N \in \{0, 1\}\})$ and $\alpha \neq 1$. 
\end{corollary}

This was first proven by \cite[Corollary~2.1]{GL90} on the Heisenberg group, and later generalised to a stratified Lie group independently by \cite[Corollary~3.9]{Amb05} and \cite[Theorem~4.1]{Kom05}.

\begin{proof}
By direct computation, $N$ verifies \eqref{qcond} with constant $Q = Q(\bbG)$ for $L = \Delta_\bbG$. 
\end{proof}

Although the homogeneous dimension $Q(\bbG)$ is strictly larger than the horizontal dimension $n_0 \vcentcolon= \dim(\mf{g}_0)$, unless $\bbG = (\bbR^m, +)$ is abelian and trivial, the weight in the homogeneous Hardy inequality carries the degenerate factor $\abs{\nabla_\bbG N}^2$ which in general is not bounded below; for instance, on a group of Heisenberg type (see \cite[Chapter~18]{BLU07}), it vanishes along the zero locus of the horizontal norm. In contrast, the horizontal Hardy inequality carries no such factor and, additionally, by comparability (see \cite[Corollary~5.1.4]{BLU07}), at least for $\alpha < 2$, yields a nondegenerate Hardy inequality for any homogeneous norm $\varrho$ which are a special class of functions defined on $\bbG$ as follows. First, define the family $(\delta_\lambda)_{\lambda > 0}$ of dilations
\[ \delta_\lambda(\xi) = \delta_\lambda(x_0, \cdots, x_{r-1}) = (\lambda x_0, \cdots, \lambda^rx_{r-1}) \]
associated to the stratification $\mf{g} = \otimes_{i=0}^{r-1} \mf{g}_i$ and $\bbG \cong \otimes_{i=0}^{r-1} \bbR_i \ni (x_0, \cdots, x_{r-1})$. A homogeneous quasinorm is a smooth (almost everywhere) function $\varrho: \bbG \rightarrow \bbR_{\geq 0}$ symmetric about the origin $\varrho(-\epsilon) = \varrho(\epsilon)$ and $1$-homogeneous $\varrho \circ \delta_\lambda = \lambda\varrho$ with respect to $(\delta_\lambda)_{\lambda > 0}$. If $\varrho$ is also positive definite, i.e. $\varrho(\xi) = 0$ if and only if $\xi = 0$, then $\varrho$ is called a homogeneous norm. We refer the reader to \cite[Chapter~5]{BLU07} for more details. 

\begin{corollary}
Let $\bbG$ be a stratified Lie group of horizontal dimension $n_0$ and let $\varrho$ be a homogeneous norm on $\bbG$. Define $\kappa(\varrho) = \inf\{\tau > 0 \mid \varrho \leq \tau N\}$. Then 
\[ \int_\bbG \frac{f^2}{\varrho^2}d\xi \leq \min\left(4, \left(\frac{2}{n_0 - 2}\right)\right)^2 \kappa^2(\abs{x_0})\kappa^2(\varrho) \int_\bbG \abs{\nabla_\bbG f}^2d\xi \]
holds for $f \in C_c^\infty(\bbG \setminus \{0\})$. 
\end{corollary}

\begin{proof}
Firstly, note $\kappa(\varrho)$ is well-defined by \cite[Corollary~5.1.4]{BLU07} which asserts comparability of all homogeneous norms, and $\kappa(\abs{x_0})$ is well-defined since $\abs{x_0}$, although not a homogeneous norm itself (but only a quasinorm), is controlled by a ``euclidean-esque" norm (see \cite[p.229]{BLU07}) and therefore by $N$. If $n_0 \geq 3$ then consider $\psi = \psi_\epsilon \vcentcolon= \abs{x_0} + \epsilon N$ where $\epsilon > 0$. It can be shown $\psi$ verifies a lower bound \eqref{qcondlower} with constant $Q = n_0 + o(\epsilon)$ for $\Delta_\bbG$ and, since $\psi_\epsilon \leq (\kappa(\abs{x_0}) + \epsilon)N$, it follows from Corollary \ref{cor:upper} that
\begin{align*}
    \int_\bbG \frac{f^2}{N^2}d\xi &\leq (\kappa(\abs{x_0}) + \epsilon)^2 \int_\bbG \frac{f^2}{\psi^2}d\xi \vphantom{\left(\frac{2}{Q - 2}\right)^2\left(\frac{1 + \epsilon}{1 - 2\epsilon + \epsilon^2}\right)^2\int_\bbG \abs{\nabla_\bbG f}^2d\xi} \\
    &\leq \frac{(\kappa(\abs{x_0}) + \epsilon)^2}{1 - 2\epsilon + \epsilon} \int_\bbG \frac{\abs{\nabla_\bbG \psi}^2}{\psi^2}f^2d\xi \vphantom{\left(\frac{2}{Q - 2}\right)^2\left(\frac{\kappa(\abs{x_0}) + \epsilon}{1 - 2\epsilon + \epsilon^2}\right)^2\int_\bbG \abs{\nabla_\bbG f}^2d\xi} \\
    &\leq \left(\frac{2}{Q - 2}\right)^2\left(\frac{\kappa(\abs{x_0}) + \epsilon}{1 - 2\epsilon + \epsilon^2}\right)^2\int_\bbG \abs{\nabla_\bbG f}^2d\xi.
\end{align*}
The desired conclusion follows by taking $\epsilon \rightarrow 0^+$. If $n_0 \leq 2$ then replace $\abs{x_0}$ in $\psi$ with $\abs{x_{0, 1}}$ where $x_{0, 1}$ is the first horizontal coordinate and appeal to Corollary \ref{cor:lower} instead. 
\end{proof}

In particular, this implies a Hardy inequality for $\varrho$ the Carnot-Carath\'eodory distance $d_{\text{CC}}$ on $\bbG$ defined by
\[ d_{\text{CC}}(x) \vcentcolon= \sup_{\abs{\nabla_\bbG f}^2 \leq 1} f(x) - f(0), \quad x \in \bbG. \]
On the Heisenberg group $\bbG = \bbH^m$ this was first proved by \cite[Theorem~1]{Leh17}, but without a description of the optimal Hardy constant, and later improved by \cite[Theorem~1]{FP21} which established 
\[ \int_{\bbH^m} \frac{f^2}{d_{\text{CC}}^2}d\xi \leq \frac{1}{m^2}\int_{\bbH^m} \abs{\nabla_{\bbH^m} f}^2d\xi, \]
that is the optimal Hardy constant is not larger than $\frac{1}{m^2}$ which coincides with $(\frac{2}{(2m+2)-2})^2 = (\frac{2}{Q(\bbH^m) - 2})^2$. Note the weighted Hardy inequality \eqref{hardygeneral} for $d_{\text{CC}}$ can also be obtained for $\alpha$ small enough, by appealing to Corollary \ref{cor:lower}; this follows from the upper bound estimate \cite[Theorem~6.1]{HZ09} for $\bbH^m$ and which holds more generally for groups of Heisenberg type. 

\subsubsection{Half-space inequalities} 
Although the dimension $n_0$ in the statement of the horizontal Hardy inequality satisfied $n_0 \geq 2$, in fact $n_0 = 1$ is perfectly acceptable in the theory of the previous chapter, but is somewhat ``special" since it corresponds to the class of half-space inequalities. Indeed, the functions $\psi(\xi) = \abs{z}$ where $z$ is any coordinate in any stratum in a stratified Lie group $\bbG$ verify \eqref{qcond} with constant $Q = 1$, since $\Delta_\bbG$ has no zeroth order (multiplicative) components and so annihilates $\psi$. That the Hardy inequality holds on a half-space is just the condition that $f$ be compactly supported away from the zero locus of $\psi$, since to be compactly supported away from the hyperplane $\{\abs{z} = 0\}$ amounts to being compactly supported on either side of it. 

\begin{corollary}
Let $\bbG$ be a stratified Lie group, $z$ any coordinate in any stratum in $\bbG$. 
\[ \int_\bbG z^\alpha \frac{\abs{\nabla_\bbG z}^2}{z^2}f^2d\xi \leq \left(\frac{2}{\alpha - 1}\right)^2 \int_\bbG z^\alpha \abs{\nabla_\bbG f}^2d\xi \]
holds for $f \in C_c^\infty(\bbG \setminus \{\abs{z} = 0\})$ and $\alpha$ such that $\alpha \neq 1$, and 
\[ \int_\bbG z \log^\alpha z\frac{\abs{\nabla_\bbG z}^2}{z^2\log^2 z}f^2d\xi \leq \left(\frac{2}{\alpha - 1}\right)^2 \int_\bbG z \log^\alpha z \abs{\nabla_\bbG f}^2d\xi \] 
holds for $f \in C_c^\infty(\bbG \setminus \{\abs{z} \in \{0, 1\})$ and any $\alpha \neq 1$. 
\end{corollary}

\begin{proof}
The coordinate functions verify \eqref{qcond} with constant $Q = 1$ for $L = \Delta_\bbG$.
\end{proof}

For instance, taking a coordinate in the first and second stratum on the Heisenberg group $\bbG = \bbH^m$ respectively recovers \cite[Theorem~2.1]{Lar16} and \cite[Theorem~1.1]{LY08}, and taking $\psi = \abs{x}$ on the line recovers the original Hardy inequality 
\[ \int_{\bbR_{>0}} \frac{f^2}{x^2}dx \leq 4\int_{\bbR_{>0}} (f')^2dx, \quad f \in C_c^\infty(\bbR_{>0}) \]
first proved by \cite[Equation~4]{Har20}.

\subsubsection{Weighted inequalities}\label{subsec4.1.1}
In proving the weighted Hardy inequalities through the operator $L_\omega = \omega L + \Gamma(\omega, -)$ we argued in \hyperref[s3.4]{\S3.4} if $\psi$ verifies \eqref{qcond} for $L$ with constant $Q \in \bbR$ then it suffices to study the quantity $\smash{\frac{\psi\Gamma(\omega, \psi)}{\omega\Gamma(\psi)}}$ to understand if $\psi$ verifies \eqref{qcond} for $L_\omega = \omega Lf + \Gamma(\omega, -)$, and that for $\omega = \psi^\alpha$ this quantity is exactly equal to $\alpha \in \bbR$. It turns out $\omega$ need not always be a power of $\psi$. For instance, on a $H$-type group $\bbG$ there are formulas for
\[ \nabla_\bbG \abs{x_0} \cdot \nabla_\bbG N = \frac{\abs{x_0}^3}{N^3} \Mand \abs{\nabla_\bbG N}^2 = \frac{\abs{x_0}^2}{N^2}, \]
(see \cite[pp.~51,~83]{Ing10}) from which it follows with $\psi = N$ and $\omega = \abs{x_0}^\alpha$ that 
\[ \frac{\psi \Gamma(\omega, \psi)}{\omega \Gamma(\psi)} = \frac{N \nabla_\bbG \abs{x_0}^\alpha \cdot \nabla_\bbG N}{\abs{x_0}^\alpha \abs{\nabla_\bbG N}^2} = \alpha \]
Since $N$ verifies \eqref{qcond} with constant $Q = Q(\bbG)$ for $L = \Delta_\bbG$, it follows $N$ verifies \eqref{qcond} with constant $Q = Q(\bbG) + \alpha$ for $L = L_\omega$ and, in fact, since $\Gamma$ is a derivation, this can be generalised to the case $\omega = \abs{x_0}^\alpha N^\beta$.

\begin{corollary}\label{prop:weightedN}
Let $\bbG$ be a $H$-type group. Then 
\[ \int_{\bbG} \abs{x_0}^\alpha N^\beta \frac{\abs{\nabla_\bbG N}^2}{N^2}f^2d\xi \leq \left(\frac{2}{Q(\bbG) + \alpha + \beta - 2}\right)^2 \int_{\bbG} \abs{x_0}^\alpha N^\beta \abs{\nabla_\bbG f}^2d\xi \]
holds for $f \in C_c^\infty(\bbG \setminus \{\abs{x_0} = 0\})$ and $\alpha, \beta$ such that $Q(\bbG) + \alpha + \beta \neq 2$, and 
\[ \int_{\bbG} \abs{x_0}^\alpha N^\beta \log^\gamma N \frac{\abs{\nabla_\bbG N}^2}{N^2\log^2N}f^2d\xi \leq \left(\frac{2}{\gamma-1}\right)^2 \int_{\bbG} \abs{x_0}^\alpha N^\beta \log^\gamma N \abs{\nabla_\bbG f}^2d\xi \]
holds for $f \in C_c^\infty(\bbG \setminus \{\abs{x_0} \in \{0, 1\}\})$ and $\alpha, \beta, \gamma$ such that $Q(\bbG) + \alpha + \beta = 2$ and $\gamma \neq 1$. 
\end{corollary}

One can also find inequalities where the roles of $\abs{x_0}$ and $N$ are reversed, but the inequalities then fall under the scope of Corollaries \ref{cor:lower}, \ref{cor:upper}, \ref{cor:loglower}, \ref{cor:logupper}, since 
\[ \frac{\abs{x_0}\Gamma(N^\beta, \abs{x_0})}{N^\beta \Gamma(\abs{x_0}, \abs{x_0})} = \frac{\abs{x_0}\nabla_\bbG N^\beta \cdot \nabla_\bbG \abs{x_0}}{N^\beta \abs{\nabla_\bbG \abs{x_0}}^2} = \beta \frac{\abs{x_0}^4}{N^4} \]
furnishes either a laplacian lower or upper bound depending on the sign of $\beta$. Similarly, mixed weight Hardy inequalities involving the Carnot-Carath\'eodory distance are possible. 

\subsubsection{Radial inequalities}
It turns out that the euclidean Hardy inequality continues to hold even if $\abs{\nabla f}^2$ is replaced with $\abs{\nabla \abs{x} \cdot \nabla f}^2 \leq \abs{\nabla f}^2$, that is 
\[ \int_{\bbR^m} \frac{f^2}{\abs{x}^2}dx \leq \left(\frac{2}{m-2}\right)^2 \int_{\bbR^m} \abs{\nabla \abs{x} \cdot \nabla f}^2dx =\vcentcolon \left(\frac{2}{m-2}\right)^2 \int_{\bbR^m} \abs{\partial_rf}^2dx. \] 
This inequality is called a radial Hardy inequality in the literature (see, for instance, \cite{BE11, Ges19}), and $\partial_r = \nabla \abs{x} \cdot \nabla$ is usually called the radial derivative. The study and proof of such inequalities can be replicated in our setting with additional work. The basic idea relies on the fact the map $f \mapsto \nabla \abs{x} \cdot \nabla f = \Gamma(\abs{x}, f)$ for $\Gamma$ the carr\'e du champ of the euclidean laplacian is a derivation and so its square is a diffusion. 

\begin{theorem}\label{thm:radialhardy}
Let $0 \leq \psi \in C^\infty(X)$ verify \eqref{qcond} with constant $Q \in \bbR$ for $L$, assume $\Gamma(\psi, \Gamma(\psi)) = 0$. Then
\begin{equation}\label{radialhardy}
    \int \psi^\alpha \frac{\Gamma(\psi)^2}{\psi^2}f^2d\mu \leq \left(\frac{2}{Q + \alpha - 2}\right)^2 \int \psi^\alpha \Gamma(\psi, f)^2d\mu
\end{equation}
holds for $f \in C_c^\infty(X \setminus \{\psi = 0\})$ and $\alpha$ such that $Q + \alpha \neq 2$, and 
\begin{equation}\label{radialloghardy} 
    \int \psi^{2 - Q} \log^\alpha\psi \frac{\Gamma(\psi)^2}{\psi^2 \log^2\psi} f^2d\mu \leq \left(\frac{2}{\alpha-1}\right)^2 \int \psi^{2 - Q} \log^\alpha\psi \Gamma(\psi, f)^2d\mu
\end{equation}
holds for $f \in C_c^\infty(X \setminus \{\psi \in \{0, 1\}\})$ and $\alpha \neq 1$. 
\end{theorem}

\begin{proof}
The starting point is the operator $L_\psi = -Z_\psi^*Z_\psi$ where $Z_\psi = \Gamma(\psi, - )$ and $Z_\psi^*$ is the adjoint of $Z_\psi$ with respect to $\mu$. Direct computation gives 
\[ \int g(Z_\psi f)d\mu = \int g\Gamma(\psi, f)d\mu = -\int f\Gamma(\psi, g)d\mu - \int fg(L\psi)d\mu =\vcentcolon \int f(Z_\psi^*g)d\mu \]
for $Z_\psi^* = -Z_\psi - L\psi$ from which it follows $L_\psi = Z_\psi^2 + (L\psi)Z_\psi$. By construction, $L_\psi$ has carr\'e du champ $\Gamma_\psi(f) = (Z_\psi f)^2$ and reversible measure $\mu$. 

Note
\[ \frac{\psi L_\psi\psi}{\Gamma_\psi(\psi)} = \frac{\psi(Z_\psi^2\psi + (L\psi)Z_\psi\psi)}{(Z_\psi\psi)^2} = \frac{\psi \Gamma(\psi, \Gamma(\psi))}{\Gamma(\psi)^2} + \frac{\psi L\psi \Gamma(\psi)}{\Gamma(\psi)^2} = Q - 1 \]
from which it follows $\psi$ verifies \eqref{qcond} with constant $Q \in \bbR$ for $L_\psi$. 
\end{proof}

We can easily recover the radial Hardy inequality in the euclidean setting, and more generally in the subelliptic setting with the horizontal norm $\abs{x_0}$ and the Kor\'anyi-Folland gauge $N$ playing the role of $\psi$. The latter however requires the analysis to occur on a $H$-type group so that the secondary condition $\Gamma(\psi, \Gamma(\psi)) = 0$ can be explicitly verified.  

\begin{corollary}\label{cor:radialhardyx}
Let $\bbG$ be a stratified Lie group of horizontal dimension $n_0 \geq 3$. Then 
\[ \int_\bbG \frac{f^2}{\abs{x_0}^2}d\xi \leq \left(\frac{2}{n_0-2}\right)^2 \int_\bbG \abs{\nabla_\bbG \abs{x_0} \cdot \nabla_\bbG f}^2d\xi \] 
holds for $f \in C_c^\infty(\bbG \setminus \{\abs{x_0} = 0\})$. 
\end{corollary}

\begin{proof}
Taking $\alpha = 0$ in \eqref{radialhardy} with $L = \Delta_\bbG$ and $\psi = \abs{x_0}$, the result follows since $\Gamma(\psi) \equiv 1$. 
\end{proof}

\begin{corollary}\label{cor:radialhardyN}
Let $\bbG$ be a $H$-type group. Then
\[ \int_\bbG \frac{\abs{\nabla_\bbG N}^4}{N^2}f^2 d\xi \leq \left(\frac{2}{Q(\bbG) - 2}\right)^2 \int_\bbG \abs{\nabla_\bbG N \cdot \nabla_\bbG f}^2d\xi \] 
holds for $f \in C_c^\infty(\bbG \setminus \{0\})$. 
\end{corollary}

\begin{proof}
Taking $\alpha$ and $L$ as before with $\psi = N$, it suffices to show $\Gamma(N, \Gamma(N)) = 0$. This can be done by using the formulas in \hyperref[subsec4.1.1]{\S4.1.1}.
\end{proof}

\begin{remark}
The condition $\Gamma(\psi, \Gamma(\psi))$ is a somewhat generic one in that a stratified Lie group $\bbG$ for which its Kor\'anyi-Folland gauge $N$ satisfies $\Gamma(N, \Gamma(N)) = 0$ is called polarisable \cite[Definition~2.12]{BT02}. It should be noted in fact our condition is slightly stronger as a group is polarisable if and only if $N$ is $\infty$-harmonic, but this follows from the fact $\Gamma(N, \Gamma(N))$ vanishes. Polarisable groups are somewhat far and few between; the $H$-type groups were the first and only examples \cite[Corollary~5.6]{BT02}, until the recent work of \cite[Theorem~5.2]{Bie19}. Nonetheless they make appearances in, for instance, the study of the $p$-sublaplacian.
\end{remark}

An alternate version of the radial Hardy inequality can be obtained by a different but very similar approach. Recalling \S4.1, suppose the Lie algebra $\mf{g}$ of $\bbG$ admits the stratification $\mf{g} = \oplus_{i=0}^{r-1} \ \mf{g}_i$ and induces a coordinate system on $\bbG \cong \otimes_{i=0}^{r-1} \ \bbR_i$ where $\bbR_i \cong \bbR^{\dim(\mf{g}_i)}$. With $\nabla_{\bbR_i}$ denoting the euclidean gradient along the $i$-th stratum, we define the dilation operator on $\bbG$ as 
\[ D_\bbG f(\xi) \vcentcolon= \sum_{i=0}^{r-1} ix_i \cdot \nabla_{\bbR^i} f(\xi). \]
For instance, if $\bbG = \bbR^m$ is abelian and trivial, then $D_\bbG = x \cdot \nabla$, and if $\bbG = \bbH^1 \cong \bbR_x^1 \times \bbR_y^1 \times \bbR_z^1$ is the $3$-dimensional Heisenberg group, then $D_\bbG = \sum_{i=1}^m x_i\partial_{x_i} + y_i\partial_{y_i} + 2z\partial_z$. That is, $D_\bbG$ is a homogeneous analogue of $x \cdot \nabla$ where each component $x_i \cdot \partial_{x_i}$ is weighted according to the stratum to which $x_i$ belongs. The following result was first proven by \cite[Theorem~2.1.1]{RS19} where the square of the dilation operator is called Euler operator. Recall we introduced the notion of a homogeneous norm and quasinorm in \hyperref[sec4.1]{\S4.1}.

\begin{theorem}
Let $\bbG$ be a stratified Lie group and let $0 \leq \psi \in C^\infty(\bbG)$ be a homogeneous quasinorm on $\bbG$. Then 
\begin{equation}\label{Dradialhardy}
    \int_\bbG \psi^\alpha f^2d\xi \leq \left(\frac{2}{Q(\bbG) + \alpha}\right)^2 \int_\bbG \psi^\alpha (D_\bbG f)^2d\xi
\end{equation}
holds for $f \in C_c^\infty(\bbG \setminus \{\psi = 0\})$ and $\alpha$ such that $Q(\bbG) + \alpha \neq 0$, and 
\begin{equation}\label{Dradialloghardy}
    \int_\bbG \psi^{-Q(\bbG)} \log^\alpha\psi \frac{f^2}{\log^2\psi}d\xi \leq \left(\frac{2}{\alpha-1}\right)^2 \int_\bbG \psi^{-Q(\bbG)} \log^\alpha\psi (D_\bbG f)^2d\xi
\end{equation}
holds for $f \in C_c^\infty(\bbG \setminus \{\psi \in \{0, 1\}\})$ and $\alpha \neq 1$. 
\end{theorem}

\begin{proof}
As previously, the starting point is the operator $L_\bbG = -D_\bbG^*D_\bbG$ where $D_\bbG$ is the dilation operator and $D_\bbG^*$ is its adjoint with respect to the Lebesgue measure $d\xi$ on $\bbG$. To compute the adjoint, note $D_\bbG$ can be rewritten in the following way: disregard the stratification $\xi = (x_0, \cdots, x_{r-1}) \in \bbR_0 \times \cdots \bbR_{r-1} \cong \bbG$ of the coordinates $\xi \in \bbG$ and instead write $\xi = (y_1, \cdots, y_m) \in \bbR^m \cong \bbG$ for $m = \dim(\bbG)$. Then $D_\bbG = \sum_{j=1}^m c_jD_j$ where $D_j = y_j\partial_j$ and $c_j = k \in \{1, \cdots, r\}$ if and only if $y_j$ is a coordinate belonging to the $k$-th stratum. Since $D_j$ has adjoint $-D_j - 1$ with respect to Lebesgue measure, it follows 
\[ D_\bbG^* = -\sum_{j=1}^m c_j(D_j + 1) = -D_\bbG - Q(\bbG) \]
and therefore $L_\bbG = D_\bbG^2 + Q(\bbG)D_\bbG$. By construction, $L_\bbG$ has carr\'e du champ $\Gamma_\bbG(f) = (D_\bbG f)^2$ and reversible measure $d\xi$, and so now we play the same game with $\psi$ and $L_\bbG$. If $\psi$ is a homogeneous quasinorm, then $D_\bbG\psi = \psi$ by \cite[Proposition~1.3.1]{RS19}. Thus
\[ \frac{\psi L_\bbG \psi}{\Gamma_\bbG(\psi)} = \frac{\psi D_\bbG^2 \psi}{(D_\bbG \psi)^2} + \frac{\psi Q(\bbG) D_\bbG \psi}{(D_\bbG \psi)^2} = \frac{\psi^2}{\psi^2} + Q(\bbG)\frac{\psi^2}{\psi^2} = Q(\bbG) + 1, \]
i.e. $\psi$ verifies \eqref{qcond} with constant $Q(\bbG) + 2$ for $L$.
\end{proof}

\begin{remark}
Note in the euclidean setting $D = x \cdot \nabla$ recovers the results of Theorem \ref{thm:radialhardy}; indeed, $(D_{\bbR^m}f)^2 = \abs{x}^2 \abs{\nabla\abs{x} \cdot \nabla f}^2$, so taking $\psi = \abs{x}$ and replacing $\alpha$ with $\alpha + 2$ we see that \eqref{radialhardy} and \eqref{radialloghardy} coincide respectively with \eqref{Dradialhardy} and \eqref{Dradialloghardy}. 
\end{remark}

\subsection{Some more examples}
We remark the analysis of the previous section requires no explicit description of the geometry of the underlying space, which is instead expressed through the laplacian comparison \eqref{qcond}. For instance, Hardy inequalities, first appearing respectively in \cite{Amb04a} and \cite{Amb05}, can be established with respect to the Grushin and Heisenberg-Greiner operators. These operators generalise respectively the euclidean laplacian and Heisenberg sublaplacian, and they enjoy analogues of the horizontal norm and Kor\'anyi-Folland gauge together with a generalisation of the homogeneous dimension to nonintegral values (done by assigning nonintegral weight to the second stratum). 

Moreover, Hardy inequalities can be established on manifolds. There are already classical results that can be recovered, for instance \cite[Theorem~1.4]{carron1997inegalites}
\[ \int_M \rho^\alpha\frac{f^2}{\rho^2}dx \leq \left(\frac{2}{C+\alpha-1}\right)^2 \int_M \rho^\alpha\abs{\nabla u}^2dx \]
for $(M, g)$ a riemannian manifold, $\rho$ satisfying $\abs{\nabla \rho} = 1$ and $\Delta \rho \geq C/\rho$ can be recovered with $L = \Delta_g$ the Laplace-Beltrami operator and $\alpha$ such that $C + 1 + \alpha > 2$ since $\psi = \rho$ verifies a lower bound \eqref{qcondlower} with $Q = C + 1$. (Note there is a discrepancy since the convention in \cite{carron1997inegalites} is that $\Delta_g$ is the nonnegative laplacian.) There is also the inequality \cite[Corollary~4.10.2]{BGL14}
\[ \int_{H^m} f^2d\mu \leq \frac{4}{(1-m)^2}\int_{H^m} \Gamma_{H^m}(f)d\mu \]
for $H^m$ the $m$-dimensional hyperbolic space realised as the Poincar\'e upper half-plane $\bbR^{m-1} \times (0, \infty)$ for $m \geq 2$, which follows from taking $L = \Delta_{H^m}$ the hyperbolic laplacian on $H^m$ and $\psi = x_m$ verifying \eqref{qcond} with $Q = 3-m$. Similarly, the well-known Hardy inequality
\[ \int_\Omega \frac{f^2}{\delta^2}dx \leq 4\int_\Omega \abs{\nabla f}^2dx \] 
for a smooth and convex domain $\Omega \sbs \bbR^m$ and $\delta(x) = \dist(x, \partial\Omega)$ the distance to the boundary can also be recovered with $L = \Delta_\Omega$ the (flat) laplacian on $\Omega$ since $\delta$ is superharmonic.